\documentclass[reqno,12pt,centertags]{amsart}
\usepackage{geometry}
\usepackage{amsmath,amsthm,amscd,amssymb,latexsym,upref,stmaryrd}
\usepackage{graphicx}
\usepackage{hyperref}
\usepackage{enumitem}
\usepackage{xcolor}
\newcommand*{\mailto}[1]{\href{mailto:#1}{\nolinkurl{#1}}}
\usepackage[numbers,sort&compress]{natbib}

\newcommand{\beq}{\begin{equation}}
	\newcommand{\eeq}{\end{equation}}
\newcommand{\ba}{\begin{align}}
	\newcommand{\ea}{\end{align}}

\renewcommand{\Re}{\text{\rm Re}}
\renewcommand{\Im}{\text{\rm Im}}

 \allowdisplaybreaks
\numberwithin{equation}{section}
\newtheorem{theorem}{Theorem}[section]
\newtheorem{lemma}[theorem]{Lemma}

\newtheorem{ip}[theorem]{Inverse Problem}
\theoremstyle{definition}

\newtheorem{proposition}[theorem]{Proposition}

\newtheorem{remark}{Remark}[section]
\newtheorem{assume}[theorem]{Assumption}
\begin{document}

	\title[A class of higher order inverse spectral problems]
	{A class of higher order inverse spectral problems}
	
	\author[A.~W.~Guan]{AI-WEI GUAN}
	\address{Department of Mathematics, School of Mathematics and Statistics, Nanjing University of
		Science and Technology, Nanjing, 210094, Jiangsu, People's
		Republic of China}
	\email{\mailto{guan.ivy@njust.edu.cn}}
	
	\author[C.~F.~Yang]{CHUAN-FU Yang}
	\address{Department of Mathematics, School of Mathematics and Statistics, Nanjing University of
		Science and Technology, Nanjing, 210094, Jiangsu, People's
		Republic of China}
	\email{\mailto{chuanfuyang@njust.edu.cn}}
	
	\author[N.~P.~Bondarenko]{NATALIA P. BONDARENKO}
	\address{S.M. Nikolskii Mathematical Institute, Peoples' Friendship University of Russia (RUDN University), 6 Miklukho-Maklaya Street, Moscow, 117198, Russian Federation}
	%	\address{1.Department of Applied Mathematics and Physics, Samaraa National Research University, Moskovskoye Shosse 34, Samara, 443086, Russia}
	%	\address{2.Department of Mechanics and Mathematics, Saratov State University, Astrakhanskaya 83, Saratov, 410012, Russia}
	\email{\mailto{bondarenkonp@sgu.ru}}

	\subjclass[2020]{34A55, 34B24, 47E05}
	\keywords{Higher-order differential operators,  Distribution coefficients, Multiple spectra, Inverse spectral problem,  Uniqueness.}
	\date{\today}
	
	\begin{abstract}
		{
			In this paper, we consider the recovery of third-order differential operators from two spectra, as well as fourth-order or fifth-order differential operators from three spectra, where these differential operators are endowed with complex-valued distributional coefficients. For the case of multiple spectra, we first establish the relationship between spectra and the Weyl-Yurko matrix. Secondly, we prove the uniqueness theorem for the solution of the inverse problems. Our approach allows us to obtain results for the general case of complex-valued distributional coefficients.
		}
	\end{abstract}
	
	\maketitle
	
\section{Introduction}
This paper deals with inverse spectral problems for the $n$-th order differential operators with distribution coefficients for $n=3,4,5$. Consider the operators generated by the following differential expression for $n=2m$:
	\begin{align}\label{1}
	\ell_{2m}(y) := y^{(2m)} + \sum_{k = 1}^{m}(-1)^{k} (\tau_{k}^{(k)}(x) y^{(m-k)})^{(m-k)},
\end{align}
where $$\tau_{1},\dots,\tau_{m}\in L_2[0,1],$$
and for $n=2m+1$:
	\begin{align}\label{2}
	\ell_{2m+1}(y)\! := & y^{(2m+1)} + \!\sum_{k = 0}^{m-1}(-1)^{k}\bigg( (\tau_{k}^{(k)}(x) y^{(m-k-1)})^{(m-k)}\!+\!(\tau_{k}^{(k)}(x) y^{(m-k)})^{(m-k-1)}\bigg),
\end{align}
where $$\tau_{0},\dots,\tau_{m-1}\in L_1[0,1].$$
The derivatives in (\ref{1}) and (\ref{2})  are to be understood in the sense of distributions.

Inverse problems of spectral analysis are concerned with recovering the properties of a signal or system from its observed spectrum. Such problems arise in various fields of science and engineering including geophysics, meteorology, physics, electronics, and other applications. Inverse problem theory for differential operators with regular coefficients has yielded numerous results. In the second-order case, the differential equation
\[ -y'' + q(x) y = \lambda y, \quad x \in (0, 1), \]
has been extensively studied, with comprehensive results in \cite{FY01, Krav20, Lev87, Mar86}. In 1946, Borg proved in \cite{Bor} that the potential of the Sturm-Liouville equation is uniquely determined by two spectra corresponding to different sets of boundary conditions. Subsequently, Gelfand and Levitan provided a constructive method for solving the inverse Sturm-Liouville problem \cite{GL51}. However, higher-order differential operators are inherently more difficult to investigate than second-order ones. 

Various issues of inverse spectral theory for the third-order and the fourth-order differential operators were investigated by Barcilon \cite{Bar1, Bar2}, McLaughlin \cite{Mc1, Mc2, Mc3}, McKean \cite{McK81}, Gladwell \cite{Gla}, 
Papanicolaou \cite{Pap04}, Badanin and Korotyaev \cite{Bad}, Perera and B\"ockmann \cite{PB20} and other scholars.
Inverse problem theory for arbitrary $n$-th order differential equations of form
\begin{equation} \label{ho}
y^{(n)} + \sum_{s = 0}^{n-2} p_s(x) y^{(s)} = \lambda y, \quad n > 2,
\end{equation}
on a finite interval and on the half-line has been constructed by Yurko (see, e.g., \cite{Yur, Yur1, Yur2}). In particular, he introduced the so-called Weyl-Yurko matrix, which uniquely specifies the coefficients $\{ p_s\}_{s = 0}^{n-2}$ of equation \eqref{ho} independently on the behavior of the corresponding spectra. Furthermore, Yurko has developed the method of spectral mappings for reconstruction and for investigation of existence for solution of the inverse problem. Inverse scattering on the line for higher-order differential operators requires a different approach (see \cite{Ba}).

Following significant progress made in the case of sufficiently smooth potential functions, research in this field has gradually shifted from regular coefficients to distributions (i.e. generalized functions). For the Sturm-Liouville operators, Hryniv and Mykutyuk generalized the basic results of inverse problem theory to distributional
potentials of class $W_2^{-1}(0,1)$ in \cite{Hry1,Hry2,Hry3}.  Afterwards, Mirzoev and Shkalikov \cite{Mir1,Mir2} developed a regularization approach for higher-order differential operators with distribution coefficients. Consequently, spectral theory for operators of this kind has been actively developed in many directions (see, e.g., \cite{KMS23} and references therein).
In particular, Bondarenko has extensively studied inverse spectral problems for higher-order differential operators with distribution coefficients. In \cite{Bon2, Bon3}, uniqueness theorems have been proved for such operators on a finite interval and on the half-line. Meanwhile, a constructive approach for recovering the coefficients of a differential operator from spectral data (eigenvalues of $(n-1)$ boundary value problems and corresponding weight numbers) has been proposed in \cite{Bon1}. This method allowed Bondarenko to obtain necessary and sufficient conditions for the inverse problem solvability in \cite{Bon7}.

In this paper, we focus on a specific class of operators generated by differential expressions \eqref{1} and \eqref{2} with $[\frac{n}{2}]$ unknown coefficients. Naturally, such operators should require less spectral data for their reconstruction than \eqref{ho} with $(n-1)$ coefficients. In particular, Barcilon \cite{Bar1} considered the uniqueness of recovering the coefficients $p$ and $q$ of the differential equation
\begin{equation}\label{4444}
y^{(4)} - (p y')' + q y = \lambda y, \quad x \in [0,1],
\end{equation}
from three spectra corresponding to various boundary conditions. However, the proof of uniqueness given in \cite{Bar1} is wrong. A correct proof has been recently provided by Guan et al \cite{Guan} for the case of distribution coefficients. A related inverse problem statement was considered by McLaughlin \cite{Mc1, Mc2}. She studied the recovery of $p$ and $q$ from a single spectrum and two sequences of norming constants and obtained solvability results by requiring the existence of transformation operator. Recently, the uniqueness of solution for McLaughlin's problem has been proved by Bondarenko \cite{Bon9}. Nevertheless, the mentioned studies are concerned with simple eigenvalues.
When a spectrum is multiple, the investigation of a differential operator becomes more challenging. The case of multiple eigenvalues was considered in \cite{But1,But2,tac} for second-order non-self-adjoint Sturm Liouville operators.

In this paper, our primary objective is to prove the uniqueness of recovering the differential expressions \eqref{1} and \eqref{2} from $([\frac{n}{2}] + 1)$ spectra for $n = 3, 4, 5$. We consider the general case of complex-valued distribution coefficients and multiple eigenvalues. The only imposed restriction is the separation condition (see Assumption~\ref{def:G}) that previously appeared in the studies of Leibenzon \cite{Leib66} and Yurko \cite{Yur}. Note that the number of required spectra in our paper is less than in the study of Baranova \cite{Bara}, which deals with the recovery of $n$-th order differential operators from $(4n-6)$ spectra. Furthermore, the choice of the spectra in this paper differs from \cite{Bara}.

Let us briefly describe the methods of our study.
The differential expressions \eqref{1} and \eqref{2} are understood in terms of the approach by Mirzoev and Shkalikov \cite{Mir1, Mir2} based on quasi-derivatives.
Then, following the idea of \cite{Guan}, we
investigate the relationship of the spectra with the Weyl-Yurko
matrix. Next, we show that not the full Weyl-Yurko matrix but only its several elements are sufficient to prove the uniqueness theorem by using the method of spectral mappings under the separation condition. Relying on this result, we prove the uniqueness by several spectra. Thus, our approach allows us to understand inverse problems for differential operators of the specific form \eqref{1} and \eqref{2} in the framework of the general inverse problem theory \cite{Yur, Bon1} for higher orders.
It is worth noting that the case $n\geq6$ requires a separate investigation.

The paper is organized as follows. In Section \ref{main results}, we describe the regularization of the differential expressions \eqref{1} and \eqref{2}, introduce the necessary concepts, and present the main results (Theorems~\ref{thm:main1}, \ref{thm:main}, and \ref{thm:main2}). Section \ref{sec:prelim} contains some preliminaries and auxiliary lemmas. In Section~\ref{sec:Weyl}, we prove the uniqueness of recovering the operator from the Weyl functions. That result plays an auxiliary role in this paper but also has a separate significance.
In Section \ref{proof0}, the uniqueness theorems by several spectra are proved for differential operators generated by \eqref{1} and \eqref{2}.

	\section{Main Results}\label{main results}
	
	Let us begin with the regularization of the differential expressions \eqref{1} and \eqref{2}. Our approach is based on the method developed by Mirzoev and Shkalikov \cite{Mir1, Mir2}, which involves constructing the matrix function $F(x)=[f_{k,j}(x)]_{k,j=1}^{n}$ that is associated with the differential expression \eqref{1} or \eqref{2} using a specific rule. This rule depends on the order of the differential equation and the classes of coefficients. 
	
	We assume that the differential equation
	\begin{equation} \label{eqv}
		l_n(y) = \lambda y, \quad x \in (0, 1),
	\end{equation}
	where $\lambda$ is the spectral parameter, can be transformed into the first-order system
	\begin{equation} \label{sys}
		Y'(x) = (F(x) + \Lambda) Y(x), \quad x \in (0, 1),
	\end{equation}
	where $Y(x)$ is a column vector function of size $n$, $\Lambda$ is the $n \times n$-matrix whose entry at position $(n,1)$ is $\lambda$ and all other entries are zero, and $F(x) = [f_{k,j}(x)]_{k,j = 1}^n$ is a matrix function with the following properties:
\begin{equation*} \label{propF}
	\begin{array}{llll}
		f_{k,j}(x) \equiv 0, \quad & k + 1 < j; \qquad
		& f_{k,k+1}(x) \equiv 1, \quad & k = \overline{1,n-1}; \\
		f_{k,k} \in L_2(0,1), \quad & k = \overline{1,n}; \qquad
		& f_{k,j} \in L_1(0,1), \quad & k > j, \quad \mbox{trace}(F(x)) = 0.
	\end{array}
\end{equation*} The class of $n \times n$ matrix functions that satisfy these properties is denoted by $\mathfrak F_n$.
	
	Given any $F \in \mathfrak F_n$, we can define the quasi-derivatives as follows:
	
	\begin{equation} \label{quasi}
		y^{[0]} := y, \quad y^{[k]} = (y^{[k-1]})' - \sum_{j = 1}^k f_{k,j} y^{[j-1]}, \quad k = \overline{1,n},
	\end{equation}
and the domain
$$
\mathcal D_F = \{ y \colon y^{[k]} \in AC[0,1], \, k = \overline{0, n-1} \}.
$$
	
A matrix function $F(x) \in \mathfrak F_n$ is called \textit{an associated matrix} of the differential expression $\ell_n(y)$ if $\ell_n(y) = y^{[n]}$ for any $y \in \mathcal D_F$.  A function $y$ is called \textit{a solution} of equation \eqref{eqv} if $y \in \mathcal D_F$ and $y^{[n]} = \lambda y$ a.e. on $(0,1)$.
	
	Let us consider a differential expression of the form \eqref{1} or \eqref{2} and an associated matrix $F(x) = [f_{k,j}]_{k,j = 1}^n$. Using the quasi-derivatives \eqref{quasi}, we define the linear forms as follows:
	
	\begin{align} \label{defU}
		U_{s}(y) & := y^{[p_{s,0}]}(0) + \sum_{j = 1}^{p_{s,0}} u_{s,j,0} y^{[j-1]}(0), \quad s = \overline{1,n},\\  \label{defV}
	V_{s}(y) & := y^{[p_{s,1}]}(1) + \sum_{j = 1}^{p_{s,1}} u_{s,j,1} y^{[j-1]}(1), \quad s = \overline{1,n},
    \end{align}
	
	Here, $p_{s,a}$ takes values from the set $\{ 0, \ldots, n-1 \}$ such that $p_{s,a} \ne p_{k,a}$ for $s \ne k$, and $u_{s,j,a}$ are complex numbers, $a=0,1$. We can also introduce the matrices $U=[u_{s,j,0}]_{s,j = 1}^n$, $V=[u_{s,j,1}]_{s,j = 1}^n$, where $u_{s,j,0} := \delta_{j,p_{s,0} + 1}$, $u_{s,j,1} := \delta_{j,p_{s,1} + 1}$. when $j > p_{s,0}$, $j > p_{s,1}$. Here and below, $\delta_{j,k}$ denotes the Kronecker delta. The problem $\mathcal L$ is defined as the triple $(F(x), U, V)$.
	
	Denote by $\{ C_k(x,\lambda) \}_{k = 1}^n$ the solutions of equation \eqref{eqv} satisfying the initial conditions
	$$
	U_{s} (C_k) = \delta_{s,k}, \quad s = \overline{1,n}.
	$$
	
	Put $\Delta_{k,k}(\lambda) := \det[V_{s}(C_r)]_{s,r = k + 1}^n$ and let $\Delta_{j,k}(\lambda)$ be obtained from $\Delta_{k,k}(\lambda)$ by the replacement of $C_j$ by $C_k$. Clearly, the functions $C_k^{[s-1]}(x, \lambda)$ for $k, s = \overline{1,n}$ and each fixed $x \in [0,1]$ as well as $\Delta_{j,k}(\lambda)$ for $1 \le k \le j \le n$ are entire analytic in $\lambda$.
	
		Let $\Lambda_{jk}\;(1\leq k\leq j\leq n)$ be the set of zeros (conting with multiplicities) of the entire function $\Delta_{j,k}(\lambda)$.

\begin{lemma}[\cite{Yur,Bon2}]\label{S}
	For $1\leq k\leq j\leq n$, the multiset $\Lambda_{jk}$ coincides with the eigenvalues (counting with multiplicities) of the boundary value problem $S_{jk}$ for the differential equation \eqref{eqv} under the conditions
	$$
	U_{\xi}(y)=V_{\eta}(y)=0,\quad \xi=\overline{1,j-1},k,\quad\eta=\overline{j+1,n}.
	$$
	Moreover, the multiset $\Lambda_{jk}$ uniquely determines the corresponding characteristic function $\Delta_{jk}(\lambda)$.
\end{lemma}
	
	Let us impose the following assumption.
	
	\begin{assume} \label{def:G}
		Suppose that $\Delta_{m,m}$  and $\Delta_{m+1,m+1}$ do not have common zeros for $1 \leq m \leq n-1$.
	\end{assume}
	
	We define the class $\mathcal{W}$ as the set of all the problems $\mathcal L$ satisfying  Assumption \ref{def:G}. Throughout this paper, we assume that $\mathcal L \in \mathcal{W}$.
	
	In the classical Sturm-Liouville problems, it is known that the potential can be uniquely recovered from two spectra (see \cite{Bor}).
	This leads to the natural question: for higher-order differential operators, do we have a similar conclusion?
	
	\begin{ip} 
		For the $n$th-order differential operators \eqref{1} and \eqref{2}, is it possible to recover $m=[\frac{n}{2}]$ potentials from $(m+1)$ spectra?
	\end{ip}

In this paper, we provide a positive answer to this question for third-, fourth-, and fifth-order differential operators. In order to formulate our uniqueness theorems, we along with $\mathcal L$ consider another problem $\tilde{\mathcal{L}}=(\tilde{F}(x), {U}, {V})$, where $\tilde{F}(x)$ is a matrix function with the same structure as ${F}(x)$ but with different coefficients. We agree that, if a symbol $\xi$ denotes an object related to $\mathcal{L}$, then the symbol $\tilde{\xi}$ with tilde will denote the analogous object related to $\tilde{\mathcal{L}}$. We assume that $\mathcal{L}=(F(x),U , V)\in \mathcal{W}$ and $\tilde{\mathcal{L}}=(\tilde{F}(x), U ,V)\in \mathcal{W}$.
	
	For definiteness, we introduce the following linear forms:
	\begin{equation} \label{UV_simp}
			U_{s}(y)=y^{[n-s]}(0), \quad
			V_{s}(y)=y^{[n-s]}(1).  \quad s = \overline{1,n}.
	\end{equation}
	
	First, let us consider the third-order differential equation
	\begin{equation}\label{l_{3}}
		l_{3}(y) := y^{(3)}+(py)'+py' = \lambda y, \quad x\in (0,1),
	\end{equation}
    where $p \in L_1(0,1)$. The associated matrix has the form
    \begin{equation}\label{matr3}
	F(x)=\begin{pmatrix} 0 & 1 & 0 \\ -p & 0 & 1  \\ 0 & -p & 0\end{pmatrix},
    \end{equation}
        
	By using the matrix \eqref{matr3} and the corresponding quasi-derivatives \eqref{quasi}, we get the relation $l_{3}(y)=y^{[3]}$.

For $i\in \{1, 2\}$, denote by $\mathfrak{S}_{i}$ the spectrum of the boundary value problem for equation (\ref{l_{3}}) with the boundary conditions
	\begin{equation*}
		U_{i}(y)=0, \quad V_{2}(y)=V_{3}(y)=0.
	\end{equation*}

 First of our main results is formulated as follows:

\begin{theorem}\label{thm:main1}
	If $\mathfrak{S}_{1}=\tilde{\mathfrak{S}}_{1}$, $\mathfrak{S}_{2}=\tilde{\mathfrak{S}}_{2}$, then $p=\tilde{p}$ in $L_1(0,1)$. Consequently, the specification of the two spectra $\mathfrak{S}_{1}$, $\mathfrak{S}_{2}$ uniquely determines the coefficient $p$ of the differential expression $l_{3}(y)$.
\end{theorem}
	
	Next, consider the fourth-order differential equation
	\begin{equation}\label{l_{4}}
		l_{4}(y):= y^{(4)}-(py')'+qy = \lambda y, \quad x\in (0,1),
	\end{equation}
where $p \in W_{2}^{-1}(0, 1)$, $q \in W_{2}^{-2}(0, 1)$. 
The associated matrix has the form
\begin{equation} \label{matr4}
	F(x)=\begin{pmatrix} 0 & 1 & 0 & 0\\ -\tau_{2} & \tau_{1} & 1 & 0 \\ \tau_{1}\tau_{2} & -\tau_{1}^{2}+2\tau_{2} & -\tau_{1} & 1 \\ \tau_{2}^{2} & -\tau_{1}\tau_{2} & -\tau_{2} & 0\end{pmatrix},
\end{equation}
where $\tau_{1}'=p$, $\tau_{2}''=q$, $\tau_1, \tau_2 \in L_2(0,1)$.
By the same way with the third-order case, the relation $l_{4}(y)=y^{[4]}$ holds. 
	
For $(i,j)\in \{(1,2), (1,3), (2,3)\}$, denote by $\mathfrak{S}_{ij}$ the spectrum of the boundary value problem for equation (\ref{l_{4}}) under the boundary conditions
	\begin{equation*}
		U_{i}(y)=U_{j}(y)=0, \quad V_{3}(y)=V_{4}(y)=0.
	\end{equation*}
	
	For the fourth-order case, we obtain the following result:
	
	\begin{theorem}\label{thm:main}
		If $\mathfrak{S}_{12}=\tilde{\mathfrak{S}}_{12}$, $\mathfrak{S}_{13}=\tilde{\mathfrak{S}}_{13}$, and $\mathfrak{S}_{23}=\tilde{\mathfrak{S}}_{23}$, then $p=\tilde{p}$ in $W_2^{-1}(0,1)$ and $q=\tilde{q}$ in $W_2^{-2}(0,1)$. Consequently, the specification of the three spectra $\mathfrak{S}_{12}$, $\mathfrak{S}_{13}$, and $\mathfrak{S}_{23}$ uniquely determines the coefficients $p$ and $q$ of the differential expression $l_{4}(y)$.
	\end{theorem}
	
Furthermore, we consider the following fifth-order differential expression
\begin{equation} \label{l_{5}}
l_{5}(y)=y^{(5)}+(py'')'+(py')''+(qy)'+qy', \quad x\in (0,1),
\end{equation}
where $p \in L_1(0,1), q \in W_{1}^{-1}(0, 1)$. The associated matrix has the form
	\begin{equation}\label{matr5}
		F(x)=\begin{pmatrix} 0 & 1 & 0 & 0 & 0\\ 0 & 0 & 1 & 0 & 0\\ \sigma_{1} & -\sigma_{0} & 0 & 1 & 0 \\ 0 & 0 & -\sigma_{0} & 0 & 1 \\ 0 & 0 & -\sigma_{1} & 0 & 0\end{pmatrix},
	\end{equation}
	where $\sigma_{0}=p$, $-\sigma_{1}'=q$, $\sigma_{0}, \sigma_{1} \in L_1(0,1)$.
	For the matrix $F(x)$ defined by \eqref{matr5} and the corresponding quasi-derivatives \eqref{quasi}, the relation $l_{5}(y)=y^{[5]}$ holds. 
	
For $(h,i,j)\in \{(1,2,3), (1,2,4), (1,2,5)\}$, denote by $\mathfrak{S}_{hij}$ the spectrum of the boundary value problem for the equation 
$$
l_{5}(y)=\lambda y, \quad x\in (0,1),
$$ 
subject to the boundary conditions
\begin{equation*}
	U_{h}(y)=U_{i}(y)=U_{j}(y)=0, \quad V_{4}(y)=V_{5}(y)=0.
\end{equation*}

Then, we obtain the following uniqueness result.

\begin{theorem}\label{thm:main2}
	If $\mathfrak{S}_{123}=\tilde{\mathfrak{S}}_{123}$, $\mathfrak{S}_{124}=\tilde{\mathfrak{S}}_{124}$, $\mathfrak{S}_{125}=\tilde{\mathfrak{S}}_{125}$, then 
	$p=\tilde{p}$ in $L_1(0,1)$ and $q=\tilde{q}$ in $W_2^{-1}(0,1)$. Thus, the specification of the three spectra $\mathfrak{S}_{123},\mathfrak{S}_{124},\mathfrak{S}_{125}$ uniquely determines the coefficients $p$ and $q$ of the differential expression $l_{5}(y)$.
\end{theorem}

\begin{remark}
	We can replace the linear forms \eqref{UV_simp} by other suitable forms in \eqref{defU}, \eqref{defV}, and our approach can also be adapted.
\end{remark}

\section{Preliminaries} \label{sec:prelim}

Let us define the notation that will be frequently used in the following sections.
If for $\lambda\rightarrow\lambda_{0}$
$$A(\lambda)=\sum_{k=-q}^{p}a_{k}(\lambda-\lambda_{0})^{k}+o((\lambda-\lambda_{0})^{p}),$$
then
$$[A(\lambda)]|_{\lambda=\lambda_{0}}^{\langle k\rangle}=A_{\langle k\rangle}(\lambda_{0}):=a_{\langle k\rangle}.$$

\subsection{Weyl-Yurko matrix} \label{sec:WY}$\newline$

Denote by $\{ \Phi_k(x,\lambda) \}_{k = 1}^n$ the solutions of equation \eqref{eqv} that satisfy the boundary conditions
\begin{equation} \label{bcPhi}
U_{s}(\Phi_k) = \delta_{s,k}, \quad s = \overline{1,k}, \qquad
V_{s}(\Phi_k) = 0, \quad s = \overline{k+1,n}.
\end{equation}

For any function $y$ in the domain $\mathcal D_F$, we introduce the notation $\vec y(x) = \mbox{col} ( y^{[0]}(x), y^{[1]}(x), \ldots, y^{[n-1]}(x))$. Consider the $n \times n$-matrices $C(x, \lambda) :=
[\vec C_k(x, \lambda)]_{k = 1}^n$ and $\Phi(x, \lambda) = [\vec \Phi_k(x, \lambda)]_{k = 1}^n$.

It has been shown in
\cite{Bon2} that the following relation holds: 
\begin{equation} \label{relM}
\Phi(x, \lambda) = C(x, \lambda) M(\lambda) 
\end{equation}
where the matrix function $M(\lambda)$ is called \textit{the Weyl-Yurko matrix} of the problem $\mathcal L$. The Weyl-Yurko matrix is the main spectral characteristics in the inverse problem theory for higher-order differential operators (see, e.g., \cite{Bon2, Bon1, Yur}).  

Due to the results of \cite{Bon2}, the Weyl-Yurko matrix $M(\lambda) = [M_{j,k}(\lambda)]_{j,k = 1}^n$ is unit lower-triangular, and its non-trivial entries have the form
\begin{equation*} \label{Mjk}
M_{j,k}(\lambda) = -\frac{\Delta_{j,k}(\lambda)}{\Delta_{k,k}(\lambda)}, \quad 1 \le k < j \le n.
\end{equation*}
Hence, $M(\lambda)$ is meromorphic in $\lambda$, and the poles of the $k$-th column of $M(\lambda)$ coincide with the zeros of $\Delta_{k,k}(\lambda)$. Note that these poles can be multiple.

Using the boundary conditions \eqref{bcPhi} on $\Phi_k(x,\lambda)$, we obtain
\begin{equation}\label{Phik}
\Phi_k(x,\lambda)\!=\!(\Delta_{kk}(\lambda))^{-1}\!\det[C_{\nu}(x,\lambda),V_{k+1}(C_{\nu}),\dots,V_{n}(C_{\nu})]_{\nu=\overline{k,n}}, \quad 1\leq k\leq n.
\end{equation}

\subsection{Problems $\mathcal L$ and $\mathcal L^{\star}$}\label{sec:star}$\newline$

Following \cite[Section 2.1]{Bon1}, we define the auxiliary problem $\mathcal{L^{\star}}=(F^{\star}(x),U^{\star} , V^{\star})$, which
helps us to investigate the structure of the Weyl-Yurko matrix. Along with the
associated matrix $F \in \mathfrak F_n$, we consider the matrix function $F^{\star}(x) = [f_{k,j}^{\star}(x)]_{k,j = 1}^n$
whose entries are defined as follows:
\begin{equation*} \label{fstar}
	f_{k,j}^{\star}(x) := (-1)^{k+j+1} f_{n-j+1, n-k+1}(x).
\end{equation*}
Obviously, $F^{\star} \in \mathfrak F_n$.

Suppose that $y \in \mathcal D_F$ and $z \in \mathcal D_{F^{\star}}$, the quasi-derivatives for $y$ are defined via \eqref{quasi} by using the elements of $F(x)$, the quasi-derivatives for $z$ are defined as
\begin{equation*} \label{quasiz}
	z^{[0]} := z, \quad z^{[k]} = (z^{[k-1]})' - \sum_{j = 1}^k f^{\star}_{k,j} z^{[j-1]}, \quad k = \overline{1,n},
\end{equation*}
and
$$
\mathcal D_{F^{\star}} := \{ z \colon z^{[k]} \in AC[0,1], \, k = \overline{0, n-1} \}.
$$

Put
\begin{gather*}
	\ell_n(y) := y^{[n]}, \quad \ell_n^{\star}(z) := (-1)^n z^{[n]}.
\end{gather*}

Along with $U$, $V$, consider the matrices 
\begin{equation*} \label{defUs}
	U^{\star} := [J_0^{-1} U^{-1} J]^T,\quad V^{\star} := [J_1^{-1} V^{-1} J]^T
\end{equation*}
where the sign $T$ denotes the matrix transpose, $J_a = [(-1)^{p_{k,a}^{\star}}\delta_{k,n-j+1}]_{k,j = 1}^n$, $p_{k,a}^{\star} := n - 1 - p_{n-k+1, a}$, $a = 0, 1$, $J=[(-1)^{k+1}\delta_{k,n-j+1}]_{k,j = 1}^n$.

The matrices $U$, $V$ generate the linear forms
\begin{align*}
U_{s}^{\star}(z) & = z^{[p_{s,0}^{\star}]}(0) + \sum_{j = 1}^{p_{s,0}^{\star}} u_{s,j,0}^{\star} z^{[j-1]}(0), \quad s = \overline{1, n}, \\
V_{s}^{\star}(z) & = z^{[p_{s,1}^{\star}]}(1) + \sum_{j = 1}^{p_{s,1}^{\star}} u_{s,j,1}^{\star} z^{[j-1]}(1), \quad s = \overline{1, n}.
\end{align*}

Consider the problem  $\mathcal L^{\star} = (F^{\star}(x), U^{\star}, V^{\star})$. Following the arguments of Subsection~\ref{sec:WY}, denote by $\{ C_k^{\star}(x, \lambda) \}_{k = 1}^n$ and $\{ \Phi_k^{\star}(x, \lambda) \}_{k = 1}^n$ the solutions of equation $\ell_n^{\star}(z) = \lambda z$, $x \in (0, 1)$, satisfying the conditions
\begin{gather*} \nonumber
	U^{\star}_{s} (C_k^{\star}) = \delta_{s,k}, \quad s = \overline{1, n}, \\ \label{bcPhis}
	U_{s}^{\star}(\Phi_k^{\star}) = \delta_{s,k}, \quad s = \overline{1, k}, \qquad
	V_{s}^{\star}(\Phi_k^{\star}) = 0, \quad s = \overline{k+1,n}.
\end{gather*}
Put $C^{\star}(x, \lambda) := [\vec C_k^{\star}(x, \lambda)]_{k = 1}^n$, $\Phi^{\star}(x, \lambda) := [\vec \Phi_k^{\star}(x, \lambda)]_{k = 1}^n$.
Then, the relation 
\begin{equation} \label{relMs}
	\Phi^{\star}(x, \lambda) = C^{\star}(x, \lambda) M^{\star}(\lambda)
\end{equation}
holds, where $M^{\star}(\lambda)$ is the Weyl-Yurko matrix of the problem $\mathcal L^{\star}$. 

\begin{proposition}[\cite{Bon1}] \label{lem:M}
	The following relations hold:
	\begin{gather*} \label{MJM}
		[M^{\star}(\lambda)]^T J_0 M(\lambda) = J_0.
	\end{gather*}
\end{proposition}

The arguments above in this section are valid for any matrix $F \in \mathfrak F_n$ and boundary condition forms \eqref{defU}, \eqref{defV}. Now, proceed to the specific cases corresponding to Theorems~\ref{thm:main1},~\ref{thm:main}, and \ref{thm:main2}. 

\begin{lemma}\label{real0}
Let $F(x)$ be defined by any of the formulas \eqref{matr3}, \eqref{matr4}, and \eqref{matr5}. Suppose that the linear forms $U_s$ and $V_s$ are given by \eqref{UV_simp}. Then
	\begin{equation}\label{real01}
		M^{\star}(\lambda) = M((-1)^n \lambda).
	\end{equation}
Moreover, for $n=3$:
\begin{equation}\label{3.}
m_{21}(-\lambda)={m}_{32}(\lambda);
\end{equation}
for $n=4$:
\begin{equation}\label{4.}
	m_{43}(\lambda)={m}_{21}(\lambda), \quad m_{42}{(\lambda)}-m_{32}(\lambda)m_{21}(\lambda)+m_{31}(\lambda)=0;
\end{equation}
for $n=5$:
\begin{equation}\label{5.}
\begin{array}{l}
	m_{21}(-\lambda)=m_{54}(\lambda), \quad m_{32}(-\lambda)=m_{43}(\lambda), \\
	m_{31}(-\lambda)-m_{21}(-\lambda)m_{43}(\lambda)+{m}_{53}(\lambda)=0.
\end{array}
\end{equation}
\end{lemma}
\begin{proof}
	By calculating, we can easily obtain the relations $U^{\star}=U$, $V^{\star}=V$, and $l_{n}^{\star}(y)=(-)^ny^{[n]}$. Therefore, 
	$$
	C^{\star}(x,\lambda)=C(x,(-1)^n\lambda), \quad \Phi^{\star}(x,\lambda)=\Phi(x,(-1)^n\lambda).
	$$ 
	Consequently, it follows from \eqref{relM} and \eqref{relMs} that
	$M^{\star}(\lambda)=M((-1)^n\lambda)$. The relations \eqref{3.}, \eqref{4.}, and \eqref{5.} can be obtained directly from Proposition \ref{lem:M} and the formula \eqref{real01}.
\end{proof}				

\subsection{Asymptotics}$\newline$

Let $\lambda = \rho^n$.
Consider the sectors
\begin{equation*}
	\Gamma_{k} :=\left\{\rho\colon \frac{\pi(k-1)}{n}\;\textless\; \arg\rho\;\textless \;\frac{\pi k}{n} \right\}, \quad k=\overline{1,2n},
\end{equation*}
and the regions
\begin{equation*}
	\mathcal{G}_{\delta}=\mathcal  G_{\delta, ijk} := \{\rho\in\overline{\Gamma}_{k}:|\rho-\rho_{ij,l}|\geq\delta,\;l \geq 1\}, \quad \delta\textgreater 0,
\end{equation*}
where $\{\rho_{ij,l}\}_{l\geq 1}$ are the zeros of $\Delta_{ij}(\rho^n)$ in the $\rho$-plane.

If $\rho$ lies in a fixed sector $\Gamma=\Gamma_{k}$, we denote by $\{\omega_{j}\}_{j=1}^{n}$ the roots of equation $\omega^{n}=1$ which are numbered in the following order: 
\begin{equation*}
	\Re(\rho \omega_{1})\;\textless\;\Re(\rho \omega_{2})\;\textless\;\dots\;\textless\;\Re(\rho \omega_{n}), \quad \rho\;\in\;\Gamma.
\end{equation*}

\begin{proposition}(\cite{Bon6})\label{prof4}
	The following relations hold as 
	$|\rho|\rightarrow\infty:$
	\begin{equation*}
		\begin{array}{ll}
			\Delta_{ij}(\lambda)=O(\rho^{a_{ij}}e^{\rho s_{j}}),\quad \rho\in\overline{\Gamma},\\
			|\Delta_{ij}(\lambda)|\geq C_{\delta}|\rho|^{a_{ij}}e^{Re(\rho s_{j})},\quad\rho\in\mathcal{G}_{\delta},
		\end{array}
	\end{equation*}
	where $\lambda=\rho^n$, $C_{\delta}$ is a constant depending on $\mathcal{G}_{\delta}$,  $s_{j}=\sum\limits_{k=j+1}^{n}\omega_{k}$,
	$$a_{ij}=\sum_{k=1}^{j-1}p_{k,0}+\sum_{k=j+1}^{n}p_{k,1}+p_{i,0}-\frac{n(n-1)}{2}.$$
\end{proposition}

% NB: I removed Proposition 3.4, since after shortening the proof of Theorem 4.1 it becomes unnecessary.

\subsection{Auxiliary Lemma}

\begin{lemma}\label{entire}
	Let $f(\cdot),\tilde{f}(\cdot),g(\cdot),\tilde{g}(\cdot),h(\cdot)$ be entire functions. Suppose $\{\lambda_m\}$ are the zeros of $h(\cdot)$ with the corresponding multiplicities $\{\kappa_{m}\}$. If $f(\lambda_m)\tilde{f}(\lambda_m)\neq0$ for all $m$, and 
	\begin{equation} \label{fracfg}
	\left(\frac{g}{f}\right)^{(k)}(\lambda_n)=\left(\frac{\tilde{g}}{\tilde{f}}\right)^{(k)}(\lambda_n), \:\: \text{for all} \:\: 0\leq k\leq \kappa_{m}-1, 
	\end{equation}
	then 
	$$
	H(\cdot):=\frac{f(\cdot)\tilde{g}(\cdot)-g(\cdot)\tilde{f}(\cdot)}{h(\cdot)}
	$$ 
	is an entire function.
\end{lemma}	
\begin{proof}
	Put $F(\lambda)=f(\lambda)\tilde{g}(\lambda)-g(\lambda)\tilde{f}(\lambda)$.
	Then, we need to show that $F^{(l-1)}(\lambda_m)=0$ for $1\leq l\leq \kappa_{m}$.
	Let us prove this by induction. By \eqref{fracfg}, it is obvious that $F(\lambda_m) = 0$. Suppose that $F^{(j)}(\lambda_m)=0$ for $0\leq j\leq i\;\textless \;\kappa_{m}-1$. It follows from \eqref{fracfg} that
	 $$
	 \bigg(\frac{f\tilde{g}-g\tilde{f}}{f\tilde{f}}\bigg)^{(k)}(\lambda_m)=0, 
	 $$ which means  $$\sum_{q=0}^{i+1}\binom{i+1}{q}(f\tilde{g}-g\tilde{f})^{(i+1-q)}(\lambda_m)\bigg(\frac{1}{f\tilde{f}}\bigg)^{(q)}(\lambda_m)=0.$$
	Consequently, 
	$$
	F^{(i+1)}(\lambda_m) = (f\tilde{g}-g\tilde{f})^{(i+1)}(\lambda_m)=0.
	$$
	By induction, this yields the claim.
\end{proof}

\section{Recovery from the Weyl functions} \label{sec:Weyl}

In this section, we prove Theorem~\ref{mfun} on the unique reconstruction of the differential equation coefficients from the Weyl functions $m_{k+1,k}(\lambda)$, $k = \overline{1,n-1}$, under Assumption~\ref{def:G}. The analogous result has been obtained by Yurko \cite{Yur} for the differential equation \eqref{ho} with regular coefficients $p_s \in W_1^s[0,1]$, $s = \overline{0,n-2}$. In this paper, we only need the following theorem for the problems $\mathcal L$ generated by the differential expressions \eqref{l_{3}}, \eqref{l_{4}}, and \eqref{l_{5}} together with the linear forms \eqref{UV_simp}.
However, the general conclusion can be deduced in a similar manner. 

\begin{theorem}\label{mfun}
	If $\mathcal L$ and $\tilde{\mathcal L}$ belong to $\mathcal{W}$ and $m_{k+1,k}(\lambda)=\tilde{m}_{k+1,k}(\lambda)$ for $k=\overline{1,n-1}$,
	then $\mathcal L=\tilde {\mathcal L}$.
\end{theorem}

The proof of Theorem~\ref{mfun} is based on auxiliary lemmas. The following lemma is proved similarly to Lemma 2.5.1 from \cite{Yur}.

\begin{lemma}\label{Yurko}
Let $\mathcal L$ and $\tilde{\mathcal L}$ belong to $\mathcal{W}$. Suppose that $\lambda_{0}$ is a zero of $\Delta_{kk}$ of multiplicity $\kappa_k\geq 1$ for a certain $k$\;$(1\leq k\leq n-1)$. Then, in a neighbourhood of the point $\lambda=\lambda_0$, we have the representation
\begin{equation}\label{xi}
	\Phi_k(x, \lambda)=\xi_k(x, \lambda)+\sum_{\nu=1}^{\kappa_k} \frac{c_{\nu k}}{\left(\lambda-\lambda_0\right)^\nu} \Phi_{k+1}(x, \lambda),
\end{equation}
where the function $\xi_k(x, \lambda)$ is regular at $\lambda=\lambda_0$.
\end{lemma}

\begin{lemma} \label{lem:eqM}
Under the conditions of Theorem~\ref{mfun}, $M(\lambda) \equiv \tilde M(\lambda)$.
\end{lemma}
	
\begin{proof}
Suppose that, for a certain $k \in \{ 1, \dots, n-1\}$, a number $\lambda_0$ is a zero of $\Delta_{k,k}(\lambda)$ of multiplicity $\kappa_k$, that is, $\lambda_0 \in \Lambda_{k k}$. Then it follows from Assumption \ref{def:G} that $\lambda_0 \notin \Lambda_{k+1, k+1}$, i.e. $\Delta_{k+1, k+1}\left(\lambda_0\right) \neq 0$, and, by Lemma \ref{Yurko}, we have the representation \eqref{xi} in a neighbourhood of $\lambda=\lambda_0$. Applying the linear form $U_{k+1}$ to the both sides of \eqref{xi} and taking into account the relations $U_{k+1}\left(\Phi_k(x, \lambda)\right)=m_{k+1, k}(\lambda)$, $U_{k+1}\left(\Phi_{k+1}(x, \lambda)\right)=1$, we obtain
		$$
		m_{k+1, k}(\lambda)=U_{k+1}\left(\xi_k(x, \lambda)\right)+\sum_{\nu=1}^{\kappa_k} \frac{c_{\nu k}}{\left(\lambda-\lambda_0\right)^\nu} .
		$$
		
		Hence
		$$
		c_{\nu k}=\left[m_{k+1, k}(\lambda)\right]_{\mid \lambda=\lambda_0}^{\langle-\nu\rangle} .
		$$
		
		By virtue of $m_{k+1,k}(\lambda)=\tilde{m}_{k+1,k}(\lambda)$ for  $k=\overline{1,n-1}$, we get $c_{\nu k}=\tilde{c}_{\nu k}$. 
		
		Applying the boundary condition $U_{k+1}$ to both sides of \eqref{xi}, the following is obtained:
			\begin{equation}
				m_{k+1,i}(\lambda)=U_{k+1}(\Phi_{i}(x, \lambda))=U_{k+1}(\xi_{i}(x, \lambda))+\sum_{\nu=1}^{\kappa_i}\frac{c_{\nu i}}{\left(\lambda-\lambda_0\right)^\nu} m_{k+1,i+1}(\lambda),
		\end{equation}
	where $\kappa_i$ is the multiple of the zero $\lambda_{0}$ of $\Delta_{ii}$.
		
	Suppose that our assertion $m_{k+1,i}(\lambda)=\tilde{m}_{k+1,i}(\lambda)$ is proved for $i=j,  \;(2\leq j \leq m)$. For $i=j-1$, we have
		\begin{equation}
			m_{k+1,j\!-\!1}(\lambda)\!=\!U_{k+1}(\Phi_{j\!-\!1}(x, \lambda))\!=\!U_{k+1}(\xi_{j-1}(x, \lambda))\!\!+\!\!\!\sum_{\nu=1}^{\kappa_{j\!-\!1}}\!\frac{c_{\nu, j-1}}{\left(\lambda-\lambda_0\right)^\nu} m_{k+1,j}(\lambda).
		\end{equation}
	Therefore, the function $m_{k+1,j-1}(\lambda)-\tilde{m}_{k+1,j-1}(\lambda)$ is regular at $\lambda_{0}$, which is the zero of $\Delta_{j-1,j-1}(\lambda)$. Hence it is entire in $\lambda$. Proposition \eqref{prof4} yields the
	estimate
	\begin{equation}
		|(m_{k+1,j-1}-\tilde{m}_{k+1,j-1})(\lambda)|\leq C_{\delta}|\rho|^{j-k-2},\quad\rho\in\mathcal{G}_{\delta},\quad|\rho|\rightarrow\infty.
	\end{equation}
By Liouville’s Theorem, $m_{k+1,j-1}(\lambda)=\tilde{m}_{k+1,j-1}(\lambda)$. By induction, we have $M(\lambda)=\tilde{M}(\lambda)$.
\end{proof}	

The proof arguments of Theorem~2 from \cite{Bon2} imply the following proposition.

\begin{proposition} \label{prop:P}
Suppose that $M(\lambda) = \tilde M(\lambda)$. Then, the matrix function $\mathcal P(x, \lambda) := \Phi(x, \lambda) (\tilde \Phi(x, \lambda))^{-1}$ does not depend on $\lambda$, that is, $\mathcal P(x, \lambda) = \mathcal P(x)$. Furthermore, $\mathcal P(x)$ is a unit lower-triangular matrix and	
\begin{equation} \label{relP}
\mathcal{P}^{\prime}(x)+\mathcal{P}(x) \tilde{F}(x)=F(x) \mathcal{P}(x), \quad x \in (0,1), \quad \mathcal{P}(0)=I, 
\end{equation}
where I is the identity matrix.
\end{proposition}

\begin{proof}[Proof of Theorem~\ref{mfun}]
	
Applying Lemma~\ref{lem:eqM} and Proposition~\ref{prop:P}, we arrive at the relations \eqref{relP}, which yield the following conclusions:

(1) The third-order case. $\mathcal{P}_{31}=p-\tilde{p}=0$. Hence $\tilde{p}=p.$

(2) The fourth-order case. $\mathcal{P}_{32}=\tilde{\tau}_{1}-\tau_{1}$, $\mathcal{P}_{32}(0)=0,$ $\mathcal{P}_{42}=\tilde{\tau}_{2}-\tau_{2}$, $\mathcal{P}_{42}(0)=0,$ $\mathcal{P}_{32}'=\mathcal{P}_{42}''=0$. Hence $\tilde{\tau}_{1}={\tau}_{1},\tilde{\tau}_{2}={\tau}_{2}+cx.$

(3) The fifth-order case. $\mathcal{P}_{42}=\tilde{\sigma}_{0}-{\sigma}_{0}=0,$ $\mathcal{P}_{41}=\tilde{\sigma}_{1}-\sigma_{1}$, $\mathcal{P}_{41}(0)=0,$
$\mathcal{P}_{41}'=0$. Hence $\tilde{\sigma}_{0}={\sigma}_{0},\tilde{\sigma}_{1}=\sigma_{1}.$

Thus, in all the three cases, the distribution coefficients of the problems $\mathcal L$ and $\tilde{\mathcal L}$ coincide.
\end{proof}

\begin{remark}
In the fifth-order case, the antiderivatives $\sigma_0$ and $\sigma_1$ are uniquely recovered from the Weyl functions in $L_2(0,1)$. However, in the fourth-order case, the function $\tau_2$ is recovered uniquely up to a term $cx$, where $c$ is an arbitrary constant, while $\tau_1$ is specified uniquely.
\end{remark}

\section{Proofs of the uniqueness theorems}\label{proof0}
In this section, we prove the uniqueness theorems for inverse problems of recovering the differential operators of orders $n=3,4,5$ from several spectra.

% NB: I removed subsections in this section.

\begin{proof}[Proof of Theorem~\ref{thm:main1}]
		By Lemma \ref{S}, we have $m_{21}(\lambda)=\tilde{m}_{21}(\lambda)$. Using \eqref{3.}, we get $m_{32}(\lambda)={m}_{21}(-\lambda)$. Finally, by Theorem \ref{mfun},  we conclude that $\mathcal{L}=\tilde{\mathcal{L}}.$
	\end{proof}

\begin{proof}[Proof of Theorem~\ref{thm:main}]
	By Lemma~\ref{S} and the definition of the Weyl-Yurko matrix, we have
\begin{equation} \label{eqm}	
m_{42}(\lambda)=\tilde{m}_{42}(\lambda), \quad m_{32}(\lambda)=\tilde{m}_{32}(\lambda). 
\end{equation}

We also have $m_{43}(\lambda)=m_{21}(\lambda)$ via (\ref{4.}), so we only need to prove $m_{21}(\lambda)=\tilde{m}_{21}(\lambda)$.
	
Expanding the Laurent series at $\lambda=\lambda_{s2}$ for $m_{j2}(\lambda)$ and $m_{k1}(\lambda)$, where $j=3,4,\;k=2,3$, we obtain the following relations:
\begin{align}\label{4..}
		m_{j2}(\lambda)& =\sum_{i=-\kappa_{m}}^{\infty}m_{j2\langle i \rangle}(\lambda_{s2})(\lambda-\lambda_{s2})^i,\quad j=3,4, \\
\label{41..}
m_{k1}(\lambda)& =\sum_{i=0}^{\infty}m_{k1\langle i \rangle}(\lambda_{s2})(\lambda-\lambda_{s2})^i,\quad k=2,3.
\end{align}
where $\{\lambda_{s2}\}_{s\geq1}$ are the zeros of the function $\Delta_{22}$ with the corresponding multiplicities $\{\kappa_{s}\}_{s\geq1}$.
	
Substituting \eqref{4..}, \eqref{41..} into \eqref{4.}, we get the following relations:
\begin{equation*}
	m_{42\langle i\rangle}(\lambda_{s2})=\sum_{l=-\kappa_{m}}^{i}m_{32\langle l\rangle}(\lambda_{s2})m_{21\langle i-l\rangle}(\lambda_{s2}),\;\;i=\overline{-\kappa_{s},-1},
\end{equation*}

Comparing the coefficients and using \eqref{eqm}, we obtain
	\begin{equation*}
		m_{21}^{(j)}(\lambda_{s2})=\tilde{ m}_{21}^{(j)}(\lambda_{s2}),\quad j=\overline{0, \kappa_{m}-1}.
	\end{equation*}
	
	Additionally, we construct the function
	$$G_{1}(\lambda)={\Delta}_{11}(\lambda)\tilde{\Delta}_{21}(\lambda)-\tilde{\Delta}_{11}(\lambda){\Delta}_{21}(\lambda).$$
	By Lemma~\ref{entire}, we  see that $K_{1}(\lambda):=\dfrac{G_1(\lambda)}{\Delta_{22}(\lambda)}$ is entire in $\lambda$.
	
	By Proposition~\ref{prof4}, we obtain the following asymptotic estimates  on the ray $\arg\rho=\frac{\pi}{4}$:
	\begin{equation*}
		{\Delta}_{11}(\lambda)=O(e^\rho),\quad
		{\Delta}_{21}(\lambda)=O(\rho^{-1}e^\rho),\quad
		|{\Delta}_{22}(\lambda)|\geq C_{\delta}e^{\Im \rho+\Re \rho}.
	\end{equation*}
	
	Consequently, we obtain
	\begin{equation*}
		|K_{1}(\lambda)|\leq|\rho|^{-1}|e^{\Re\rho-\Im\rho}|=|\rho|^{-1}\rightarrow 0,\;\;\;\arg\rho=\frac{\pi}{4},\;|\rho|\rightarrow\infty.
	\end{equation*}
	Hence
	\begin{equation*}
		|K_{1}(\lambda)|\rightarrow 0,\;\;\;\arg\lambda=\pi,\;\lambda\rightarrow-\infty.
	\end{equation*}
	The entire functions ${\Delta}_{21}(\lambda)$, ${\Delta}_{22}(\lambda)$ and ${\Delta}_{11}(\lambda)$ are order of $\frac{1}{4}$ and so does the function $K_{1}(\lambda)$. Applying Phragmen-Lindel\"of's theorem \cite{BFY},  we can get $K_{1}(\lambda) \equiv 0,\;\lambda\in \mathbb{C}$. Hence $G_1(\lambda) \equiv 0,\;\lambda\in \mathbb{C}$ and so $m_{21}(\lambda)=\tilde{m}_{21}(\lambda)$.
	By Theorem \ref{mfun},  we have $\mathcal{L}=\tilde{\mathcal{L}}.$	
\end{proof}

\begin{proof}[Proof of Theorem~\ref{thm:main2}]
	By Lemma \ref{S}, we have $m_{43}(\lambda)=\tilde{m}_{43}(\lambda)$ and $m_{53}(\lambda)=\tilde{m}_{53}(\lambda)$. Using \eqref{5.}, we also conclude that $m_{43}(\lambda)=\tilde{m}_{32}(-\lambda)$. So, we only need to prove $m_{21}(\lambda)=\tilde{m}_{21}(\lambda)$.
	
	Denote $\{\lambda_{s2}\}_{s\geq1}$, $\{\lambda_{s3}\}_{s\geq1}$ are the zeros of the function $\Delta_{22}$, $\Delta_{33}$ with multiplicity $\{{\kappa_{s2}}\}_{s\geq1}$ and $\{{\kappa_{s3}}\}_{s\geq1}$ respectively. By \eqref{5.}, we have  $\{\lambda_{s2}\}_{s\geq1}=\{-\lambda_{s3}\}_{s\geq1}$ and  $\{{\kappa_{s2}}\}_{s\geq1}=\{{\kappa_{s3}}\}_{s\geq1}$.
	
	Similarly, we expand the Laurent series at $\lambda=\lambda_{s2}$ for \eqref{5.} and compare the coefficients, this yields the following relations:
	\begin{equation*}
		m_{53\langle i\rangle}(\lambda_{s2})=\sum_{l=-\kappa_{s2}}^{i}m_{43\langle l\rangle}(\lambda_{s2})m_{21\langle i-l\rangle}(\lambda_{s2}),\;\;i=\overline{-\kappa_{s2},-1},
	\end{equation*}
Hence, we have 
\begin{equation*}\label{equal11}
		m_{21}^{(j)}(\lambda_{s2})=\tilde{ m}_{21}^{(j)}(\lambda_{s2}),\;\;\;j=\overline{0, \kappa_{s2}-1}.
\end{equation*}

	Construct the function
	$$G_{2}(\lambda)={\Delta}_{11}(\lambda)\tilde{\Delta}_{21}(\lambda)-\tilde{\Delta}_{11}(\lambda){\Delta}_{21}(\lambda)$$
	Lemma \ref{entire} implies that $K_{2}(\lambda):=\dfrac{G_{2}(\lambda)}{\Delta_{22}(\lambda)}$ is entire in $\lambda$.
	
	By Proposition \ref{prof4}, we get the following asymptotic estimates on the ray $\rho > 0$:

	\begin{equation*}
		{\Delta}_{11}(\lambda)=O(e^{\rho(\cos\frac{6\pi}{5}+2\cos\frac{2\pi}{5}+1)}),
	\end{equation*}
	\begin{equation*}
		{\Delta}_{21}(\lambda)=O(\rho^{-1}e^{\rho(\cos\frac{6\pi}{5}+2\cos\frac{2\pi}{5}+1)}),
	\end{equation*}
	\begin{equation*}
		|{\Delta}_{22}(\lambda)|\geq C_{\delta}e^{\rho(2\cos\frac{2\pi}{5}+1)}.
	\end{equation*}
	Consequently, we obtain
	\begin{equation*}
		|K_{2}(\lambda)|\leq|\rho|^{-1}e^{\rho(4\cos^{2}\frac{\pi}{5}-2\cos\frac{\pi}{5}-1)}=|\rho|^{-1}\rightarrow 0,\;\;\;\rho > 0,\; \rho \to +\infty.
	\end{equation*}
	Note that
	\begin{equation*}
		|K_{2}(\lambda)|\rightarrow0,\;\;\;\lambda > 0,\;\lambda\rightarrow+\infty.
	\end{equation*}
	The entire functions ${\Delta}_{21}(\lambda)$, ${\Delta}_{22}(\lambda)$ and ${\Delta}_{11}(\lambda)$ are of order $\frac{1}{5}$ and so does the function $K_{2}(\lambda)$. Applying Phragmen-Lindel\"of's theorem \cite{BFY},  we can get $K_{2}(\lambda) \equiv 0,\;\lambda\in \mathbb{C}$. Hence $G_{2}(\lambda) \equiv 0,\;\lambda\in \mathbb{C}$ and so $m_{21}(\lambda)=\tilde{m}_{21}(\lambda)$.
	By Theorem \ref{mfun}, we have $\mathcal{L}=\tilde{\mathcal{L}}.$
	
\end{proof}

In conclusion, it is worth mentioning that the problem investigated in this study pertains to differential operators of third, fourth, and fifth order. Dealing with the orders of greater than $5$ requires a different approach, so this is a topic for future research.


\begin{thebibliography}{99}

\bibitem{Bad}Badanin, A., Korotyaev, E.L., Third-order operators with three-point conditions associated
with Boussinesq’s equation. Applicable Analysis, 2021, 100(3): 527–560.

\bibitem{Bara} Baranova, E. A., On the recovery of higher-order differential operators by means of a system of their spectra. In Doklady Akademii Nauk, Russian Academy of Sciences, 1972, 205(6): 1271-1273.

\bibitem{Bar1} Barcilon, V., On the uniqueness of inverse eigenvalue problems. Geophysical Journal International, 1974, 38(2): 287-298.

\bibitem{Bar2} Barcilon, V., On the solution of inverse eigenvalue problems of high orders. Geophysical Journal International, 1974, 39(1): 143-154.

\bibitem{Ba}Beals, R., Deift, P., Tomei, C.,  Direct and inverse scattering on the line. American Mathematical Society, 1988, 28.

\bibitem{Bon2} Bondarenko, N. P., Inverse spectral problems for arbitrary-order differential operators with distribution coefficients. Mathematics, 2021, 9(22): 2989. 

\bibitem{Bon1} Bondarenko, N. P., Reconstruction of higher-order differential operators by their spectral data. Mathematics, 2022, 10(20): 3882. 

\bibitem{Bon3} Bondarenko, N. P., Linear differential operators with distribution coefficients of various singularity orders. Mathematical Methods in the Applied Sciences, 2023, 46(6): 6639–6659.

\bibitem{Bon7} Bondarenko, N. P., Necessary and sufficient conditions for solvability of an inverse problem for higher-order differential operators. Mathematics, 2024, 12(1): 61.

\bibitem{Bon6} Bondarenko, N. P., Spectral data asymptotics for the higher-order differential operators with distribution coefficients. Journal of Mathematical Sciences, 2022, 266(5): 1-22.

\bibitem{Bon9} Bondarenko, N. P., McLaughlin's inverse problem for the fourth-order differential operator. arXiv:2312.15988.

\bibitem{Bor} Borg, G., Eine umkehrung der Sturm-Liouvilleschen eigenwertaufgabe. Acta Mathematica, 1946, 78(1): 1-96.

\bibitem{But1} Buterin, S. A., On inverse spectral problem for non-selfadjoint Sturm–Liouville operator on a finite interval. Journal of Mathematical Analysis and Applications, 2007, 335(1): 739-749.

\bibitem{BFY} Buterin, S. A., Freiling, G., Yurko V.\;A., Lectures in the theory of entire functions. Schriftenriehe der Fakultat fur Matematik, Duisbug-Essen University. SM-UDE-779, 2014.

\bibitem{But2} Buterin, S. A., Kuznetsova, M., On Borg’s method for non-selfadjoint Sturm–Liouville operators. Analysis and Mathematical Physics, 2019, 9(4): 2133-2150.

\bibitem{FY01} Freiling, G., Yurko, V., Inverse Sturm-Liouville Problems and Their Applications. Nova Science Publishers, Huntington, NY, 2001.

\bibitem{GL51}
Gel'fand, I. M., Levitan, B. M., On the determination of a differential equation from its spectral function. Izvestiya Rossiiskoi Akademii Nauk. Seriya Matematicheskaya, 1951, 15(4): 309-360.

\bibitem{Gla} Gladwell, G.M.L., Inverse Problems in Vibration. Second Edition, Solid Mechanics and Its Applications, Springer, Dordrecht, 2005, 119.

\bibitem{Guan} Guan, A. W., Yang, C. F., Bondarenko, N. P., Solving Barcilon's inverse problems by the method of spectral mappings. 	arXiv:2304.05747.

\bibitem{Hry3} Hryniv, R. O., Mykytyuk, Y. V., Half-inverse spectral problems for Sturm–Liouville operators with singular potentials. Inverse Problems, 2004, 20(5): 1423.


\bibitem{Hry1} Hryniv, R. O., Mykytyuk, Y. V., Inverse spectral problems for Sturm–Liouville operators with singular potentials. Inverse Problems, 2003, 19(3): 665-684.

\bibitem{Hry2} Hryniv, R. O., Mykytyuk, Y. V., Inverse spectral problems for Sturm-Liouville operators with singular potentials, II. Reconstruction by two spectra. In North-Holland Mathematics Studies, North-Holland, 2004, 197: 97-114.

\bibitem{KMS23}
Konechnaja, N. N., Mirzoev, K. A., Shkalikov, A. A., Asymptotics of solutions of two-term differential equations. Mathematical Notes, 2023, 113(2): 228–242.

\bibitem{Krav20} Kravchenko, V. V., Direct and Inverse Sturm-Liouville Problems. A method of solution. Springer Nature, 2020.	

\bibitem{Leib66}
Leibenson, Z. L., The inverse problem of spectral analysis for higher-order ordinary differential operators. Transactions of the Moscow Mathematical Society, 1966, 15: 78-163.

\bibitem{Lev87}
Levitan, B. M., Inverse Sturm-Liouville Problems. VSP, 1987.

\bibitem{Mar86} Marchenko, V. A., Sturm-Liouville Operators and Their Applications. Birkh\"auser, Basel, 1986.

\bibitem{McK81} McKean, H., Boussinesq's equation on the circle. Communications on Pure and Applied Mathematics, 1981, 34(5): 599-691.

\bibitem{Mc1} McLaughlin, J. R., Higher order inverse eigenvalue problems. In: Everitt, W., Sleeman, B.
(eds) Ordinary and Partial Diﬀerential Equations. Lecture Notes in Mathematics, 
Springer, Berlin, Heidelberg, 1982, 964.

\bibitem{Mc2} McLaughlin, J. R., Analytical methods for recovering coefficients in differential equations
from spectral data. Society for Industrial and Applied Mathematics Review, 1986, 28(1): 53–72.

\bibitem{Mc3} McLaughlin, J. R., Bounds for constructed solutions of second and fourth order inverse
eigenvalue problems. North-Holland Mathematics Studies, 1984, 92: 437–443.

\bibitem{Mir1} Mirzoev, K. A., Shkalikov, A. A., Differential operators of even order with distribution coefficients. Mathematical Notes, 2016, 99: 779-784.

\bibitem{Mir2} Mirzoev, K. A., Shkalikov, A. A., Ordinary differential operators of odd order with distribution coefficients. arXiv:1912.03660, 2019.

\bibitem{Pap04}
Papanicolaou, V. G.,	An inverse spectral result for the periodic Euler-Bernoulli equation. Indiana University Mathematics Journal, 2004, 53(1): 223-242.

\bibitem{PB20}
Perera, U., B\"ockmann, C., Solutions of Sturm-Liouville problems. Mathematics, 2020, 8(11): 2074.

\bibitem{tac}Tkachenko, V., Non-selfadjoint Sturm-Liouville operators with multiple spectra. In Interpolation Theory, Systems Theory and Related Topics: The Harry Dym Anniversary,  Birkhäuser Basel, 2002, 134: 403-414.  


\bibitem{Yur} Yurko, V. A., Method of spectral mappings in the inverse problem theory. VSP, 2002.

\bibitem{Yur1} Yurko, V. A., On higher-order differential operators with a singular point. Inverse Problems, 1993, 9(4): 495.

\bibitem{Yur2} Yurko, V. A., On higher-order differential operators with a regular singularity. Sbornik: Mathematics, 1995, 186(6): 901-928.		

\end{thebibliography}
\end{document}